\documentclass{amsart}
\usepackage{geometry}
\usepackage{tabularx}
\usepackage{enumitem,amsmath,amsfonts,amssymb,xcolor}
\usepackage{latexsym,amsfonts,euscript,amssymb,gensymb}
\newtheorem{lemma}{\bf Lemma}[section]

\newtheorem{thm}[lemma]{\bf Theorem}

\newtheorem{con}[lemma]{\bf Conjecture}
\newtheorem{pro}[lemma]{\bf Problem}
\newtheorem{cla}[lemma]{\bf Claim}

\newcommand{\SL}{{\operatorname{SL}}}

\DeclareMathOperator{\Aut}{Aut}

\input xy
\xyoption{all}

\title[]{On Thompson Problem}

\author{Rulin Shen}
\address{Rulin Shen. Department of Mathematics, Hubei Minzu University, Enshi, Hubei, P.R.China, 445000.}
\email{shenrulin@hotmail.com}

\author{Wujie Shi}
\address{Wujie Shi. School of Mathematics and Big Date, Chongqing University of Arts and Sciences, Chongqing, P.R.China,  402160;
    School of Mathematical Sciences, Soochow University, Suzhou, Jiangsu, P.R.China, 215006.}
\email{wjshi@suda.edu.cn}

\author{Feng Tang}
\address{Feng Tang. School of Mathematical and Statistics, Changshu Institute of Technology, Changshu, Jiangshu, P.R.China, 215500.}
\email{tangff-04@163.com}

\thanks{\scriptsize Project supported by the NNSF of China (Grant No.10571128 and 10871032) and the foundation of Educational Department of Hubei Province in China (Grant No.
 Q20092905 and Q20111901) . }

\date{}

\keywords{Thompson problem, Same order elements, Disconnected prime graph }

\begin{document}

\setlength{\parskip}{2mm}

\maketitle

\begin{footnotesize}
\quad \qquad This paper is published in Chinese in the journal Scientia Sinica Mathematica, no.40(2010), \\ \indent \quad \qquad pp.533-537 and
no.41(2011), pp.933-938(corrigendum). Translate it to English is helpful for \\ \indent \quad \qquad  the citing some conclusions in this paper.
\end{footnotesize}


\begin{abstract}
In 1987, the second author of this paper reported his conjecture, all
finite simple groups $S$ can be characterized uniformly using the
order of $S$ and the set of element orders in $S$, to Prof. J. G. Thompson.
In their communications, Thompson posed his problem
about the judgment of solvability of finite groups $G$. In this
paper we give a positive answer for Thompson's problem if the prime
graph of $G$ is not connection.
\end{abstract}

\section{Introduction}

Given a finite group $G$, how to determine the solvability of $G$?
Sometimes we can judge whether a group $G$ is solvable using the famous
Odd Order Theorem, which states that all finite groups of odd order are solvable.
The following conjecture was posed in a letter from the
second author of this article to Professor Thompson in 1987:

\begin{con}
Let $G$ be a finite group and $S$ a finite simple group. Then $G \cong S$
if and only if $|G|=|S|$ and $\pi_e(G)= \pi_e(S)$, where $\pi_e(G)$ denotes the
set of element orders in $G$.
\end{con}

Based on previous work, the conjecture mentioned above has been proven in \cite{VAV},
thus establishing the conjecture as a theorem.
In their communications, Thompson posed the following problem \ref{thompson}
(see \cite{TJG}). Let $G$ be a finite group and let $d \geqslant 1$ be an integer, set $G(d)=
\{x \in G | x^d = 1 \}$. We say that two groups $G_1$ and $G_2$ are of $the~same~order~type$ if and only
if $|G_1(d)|=|G_2(d)|$, $d=1$, $2, \cdots $.

\begin{pro}[Thompson Problem]\label{thompson}
Suppose $G_1$ and $G_2$ are groups of
the same order type. Suppose also that $G_1$ is solvable. Is it true that $G_2$ is also necessarily solvable?
\end{pro}

For groups $G$ of even order, we can not determine the
solvability of $G$ with the order of $G$, but we may judge it
using the same order type of $G$ if the answer of above problem is positive.
In Thompson's private letter he pointed out that:``I have talked with several mathematicians concerning groups of the same order type. The problem arose initially in the study of algebraic number fields, and is of considerable interest".

Let $\pi(n)$ denote the set of prime divisors of a natural number $n$ and write $\pi(G)=\pi(|G|)$.
Using set $\pi(G)$, we can define the prime graph of $G$ (see \cite{WJS}): $\pi(G)$ is the set
of vertices, vertices $r$ and $s$ in $\pi(G)$ are joined by edge if and only if $rs\in \pi_e(G)$.

In this paper we prove the following theorem:

\begin{thm}\label{main}
Let $G$ and $F$ be any two finite groups of the same order type. If $F$
is solvable and its prime graph is disconnected, then $G$ is solvable.
\end{thm}

\section{Preliminaries and Lemmas}

Let $M_G(n)$ be the set of elements of $G$ with order $n$, then the elements of $G$ can be partitioned into some $M_G(n)$, that is,
 $$G=\bigcup_{n\in\pi_e(G)}M_G(n),$$
 where $\pi_e(G)$ is the spectrum of $G$, which is the set of element orders of $G$.
Since every element of order $n$ must be contained in some cyclic subgroup of order $n$, we have
$|M_G(n)|=v_n(G)\phi(n)$, where $v_n(G)$ is the number of cyclic
subgroups of order $n$, called $cyclic~degree~of~order~n$, and
$\phi$ is Euler totient function.
So we can obtain the following
equation, called $order~equation$ of $G$,
$$|G|=\sum_{n\in\pi_e(G)}v_n(G)\phi(n).$$
For convenience, we write order equation of $G$ as $Ord(G)$.  Moreover, if $Ord(G)=Ord(F)$, then $\pi_e(G)=\pi_e(F)$ and $v_d(G)=v_d(F)$ for
any $d\in \pi_e(G)$. It is easy to see that this equality is equivalent to
the same order type of Thompson problem. So we can use same
order equation instead of same order type.  Thompson problem has also quoted
in the book Unsolved Problems in Group Theory (see \cite [Problem 12.37] {MVD}).
It is not hard to prove that if $G$ is nilpotent then
$F$ is nilpotent and if $G$ is supersolvable then $F$ is
solvable (See \cite{XM}). The Shi's conjecture has recently been resolved (See Problem 12.39 of \cite{MVD} or \cite{S}),
and from this conjecture we can see that if $G$ is a finite nonabelian simple group, then $F$ is non-solvable.

Obviously, the set $\pi_e(G)$ forms a partially ordered set with
respect to the divisibility relation, we say $\mu(G)$ is the maximal
element in this partially ordered set. If $n$ is a natural number,
$\pi$ a set of primes, then by $\pi(n)$  the set of all prime
divisors of $n$, by $|n|_r$ we denote the $r$-part of $n$. Now we can
construct a prime graph of $G$ with the set $\pi(G)$ as follows: the
vertices are all elements of $\pi(G)$, two vertices $r, s$ are
joined by an edge if and only if $G$ contains an element of order
$rs$. This graph is called the $Gruenberg$-$Kegel$ $graph$ or
$prime$ $graph$ of $G$ and denoted by $GK(G)$. Denote the connected
components of the graph by $\{\pi_i\mid i=1,\cdots,s:=s(G)\}$ and
if the order of $G$ is even, denote the component containing 2 by
$\pi_1$, $s(G)$ is said to be the number of connected components of $G$.
Set $\mu_i(G)=\{a\in \mu(G)|\pi(a)\subseteq\pi_i(G)\}$.
Williams \cite{WJS} and Kondrat'ev \cite {KA} determined the number
of connected components of finite simple group by using some number
theory techniques. Recently, Buturlakin provided the cyclic
structures of maximal tori of finite classical simple groups (see \cite{BAA}). For finite exceptional simple groups, the cyclic
structures of maximal tori were given in article \cite{K}.
Consequently, it is easy to observe that $\mu_i(G)$ with
$i\geqslant 2$ has a unique element, denoted by $n_i(G)$, which is
the unique element of $\mu_i(G)$. The structure of groups with more
than one connected components of prime graph were given in
\cite{WJS}. It states the following.

\begin{thm}[Gruenberg-Kegel]\label{Gruenberg}
If $G$ is a finite group whose prime graph has more than one components,
then one of the following holds.
\begin{enumerate}
\item $s(G)=2$ and $G$ is a Frobenius group or $2$-Frobenius group.

\item There exists a nonabelian simple group $S$ such that
      $S\leqslant \overline{G}=G/N \leqslant \Aut(S)$, where $N$ is a maximal normal soluble
      subgroup of $G$. Furthermore, $N$ and
      $\overline{G}/S$ are $\pi_1(G)$-subgroups, $2\leqslant s(G)\leqslant s(S)$, and for every $i$ where $2\leqslant i\leqslant s(G)$, there is a $j$ with $2\leqslant j\leqslant s(S)$ such that
      $\mu_i(G)=\mu_j(S)$.
\end{enumerate}
\end{thm}

From well-known Thompson theorem we deduce that $N$ is nilpotent.
If we restrict some stronger conditions to Theorem \ref{Gruenberg}, we can obtain the following stronger result.

\begin{thm}\label{strongerversion}
Let $G$ and $F$ be two finte groups and $s(F)\geqslant 2$. If $Ord(G)=Ord(F)$, then one of the following occurs.

\begin{enumerate}
\item $s(G)=2$ and $G$ is a Frobenius group or $2$-Frobenius group.

\item There exists a nonabelian simple group $S$ such that
     $S\leqslant \overline{G}=G/N \leqslant \Aut(S)$ where $N$ and $\overline{G}/S$ are
     $\pi_1(G)$-subgroups, the graph $GK(G)$ is disconnected and $s(S)\geqslant
     s(G)$. Moreover, the following statements hold:

     \begin{enumerate}
     \item For every $i$ where $2\leqslant i\leqslant s(G)$, there is a $j$ with $2\leqslant j\leqslant s(S)$ such that $\mu_i(G)=\mu_j(S)$.

     \item Let $m>1$ be a  divisor of any element of $\mu_i(F)$ where
           $i\geqslant 2$. Then $|M_F(m)|=|M_{\overline{G}}(m)||N|$ and
           $|M_{\overline{G}}(m)|=|M_S(m)|$.
      \end{enumerate}
\end{enumerate}
\end{thm}

\begin{proof}
To prove our theorem, it is enough to prove items $(a)$ and $(b)$ of $(2)$. Let $x$ be an element of $G$ of order $m$.
Since $\pi(N)\subseteq \pi_1(G)=\pi_1(F)$ and $\pi(m)\subseteq
\pi_i(F)=\pi_i(G)$ where $i\geqslant 2$, $(m,|N|)=1$. We assert that $o(x)=o(xN)$.
In fact, let $o(xN)=n$, then $n|m$ since $(xN)^m=x^mN=N$.
Conversely, $x^n\in N$ follows from $N=(xN)^n=x^nN$, this implies $o(x^n)\mid |N|$. Since
$$o(x^n)=\frac{o(x)}{(n,o(x))}=\frac{m}{(n,m)},$$
we have $\frac{m}{(n,m)}\mid |N|$, hence $\frac{m}{(n,m)}=1$, it follows that $m|n$. Therefore, $m=n$.
On the other hand, if $G$ has an element $y$ with $o(y)\neq m$ such that $o(yN)=o(xN)$, then $o(xN)|o(y)$, namely,
$m|o(y)$, we may assume $o(y)=mq$ where $q\neq 1$ and there exists a
prime divisor $r$ of $q$ such that $r\in\pi_j(G)$ with $i\neq j$
(otherwise, $q$ is a divisor of some element of $\mu_i(G)$, we have
$o(yN)=o(y)=mq$ from above arguement, we get a contradiction). This shows that
$\pi_i(G)$ and $\pi_j(G)$ are connected in the prime graph of $G$, a
contradiction. Hence, $o(y)=m$ and $|M_{\overline{G}}(m)|\leqslant
|M_F(m)|=|M_G(m)|$. Moreover, all elements of coset
$xN$ have order $m$, so
$|M_G(m)|=|M_{\overline{G}}(m)||N|$. For the rest of the part, since
$\overline{G}$ is an almost simple group, $\overline{G}$ is an
extension of $S$ by $\overline{G}/S$. We choose a set of coset
representatives $\{y_1, y_2, \cdots, y_l\}$ of $\overline{G}/S$,
where $y_i\in \overline{G}$. Obviously, all elements of
$\overline{G}$ are of the form $y_is$ where $s\in S$ and $1\leqslant i\leqslant
l$. Since
$o(y_iS)=o(y_isS)$, we have $o(y_iS)|o(y_is)$. If $y_iS\neq S$, then by
the above condition that $\overline{G}/S$
 is a $\pi_1(G)$-subgroup, we have $\pi(o(y_iS))\subseteq
 \pi_1(\overline{G})$, so there is a prime divisor of
 $o(y_is)$ which belongs to $\pi_1(G)$. We assume $\pi(o(s))\subseteq
 \pi_i(\overline{G})$ with $i\geq 2$, then any element $x$ of $\overline{G}-S$
 with the property that $o(x)$ is not a divisor of $o(s)$ since
 $\pi_1(\overline{G})$ and $\pi_i(\overline{G})$ are not connected in
 the prime graph of $\overline{G}$. Therefore,
  $|M_{\overline{G}}(m)|=|M_S(m)|$.
\end{proof}

In order to complete the proof of Theorem \ref{main}, we need to introduce some
useful results about Frobenius group and $2$-Frobenius group.

\begin{lemma}[see \cite {MD} and \cite {CG}]\label{frobenius}
Suppose that $F$ is a Frobenius group with
kernel $K$ and complement $H$, then the following statements hold.

\begin{enumerate}
\item $K$ is nilpotent and the Sylow $p$-subgroup of $K$ is cyclic where $p$ is an odd prime number.
\item If $2\in\pi(H)$, then $K$ is abelian. Moreover, if $H$ is
      solvable then the Sylow 2-subgroup of $H$ is either a cyclic group or a
      (generalized) quaternion group $Q_{2^n}$ and if $H$ is non-solvable
      then there exists a subgroup $H_0$ with $|H:H_0|\leqslant 2$ such that
      $H_0\cong Z\times \SL_2(5)$, where $Z$ has properties that all Sylow subgroups of $Z$ are
      cyclic and $(|Z|,30)=1$.
\end{enumerate}
\end{lemma}

\begin{lemma}[see \cite {CG}]\label{2frobenius}
Suppose that $E=ABC$ is a $2$-Frobenius group, where $A$ and $AB$ are normal subgroups of $E$, $AB$ and
$BC$ are Frobenius groups with kernel $A$, $B$ and complement $B$,
$C$, respectively. Then the following statements hold.
\begin{enumerate}
\item $s(E)=t(E)=2$, $\pi_1(E)=\pi(A)\cup\pi(C)$ and $\pi_2(E)=\pi(B)$.
\item $E$ is solvable, $B$ is a Hall cyclic subgroup of $E$ with odd order and $C$ is a cyclic group.
\end{enumerate}
\end{lemma}

\begin{lemma}\label{divisible}
Let $G$ be a finite group, $E$ a $2$-Frobenius group and $Ord(G)=Ord(E)$. If there
is a nonabelian simple group $S$ such that $S\leqslant \overline{G}=G/N\leqslant \Aut(S)$,
where $N$ is the maximal normal soluble subgroup of $G$, $N$ and $\overline{G}/N$ are
$\pi_1(G)$-groups, then $|N|$ divides $|A|$ and $\pi(\frac{|A|}{|N|})\subseteq \pi(|C|)$.
\end{lemma}

\begin{proof}
Let $m>1$ be a divisor of any element of $\mu_i(E)$ where $i\geqslant 2$, from
Theorem \ref{strongerversion}(2)(b) we have $|M_E(m)|=|M_{\overline{G}}(m)|\cdot|N|$. It is easy to see that
$\mu_2(E)=\{|B|\}$ and $|M_E(|B|)|=|A|\phi(|B|)$. Since $Ord(G)=Ord(E)$, we have
$$|M_G(|B|)|=|M_E(|B|)|=|M_{\overline{G}}(|B|)||N|=v_{|B|}(\overline{G})\phi(|B|)|N|.$$
Therefore $|A|=v_{|B|}(\overline{G})\cdot |N|$, this implies that $|N|$ divides $A$.

If there is a prime number $p$ with $p\in \pi(\frac{|A|}{|N|})$ such that
$p\not\in\pi(|C|)$, then the number of $p$-elements of $E$ is
$|A|_p$. Since $Ord(G)=Ord(E)$, the number of $p$-elements
of $G$ is $|G|_p$ and hence the Sylow $p$-subgroup $P$ of $G$ is
normal. Note that $N$ is the maximal normal soluble subgroup of $G$,
so $P\leqslant N$, thus $(p, \frac{|G|}{|N|})=1$. As
$|G|=|A||B||C|$, we can obtain $(p, \frac{|A|}{|N|})=1$, which contradicts our
hypothesis.
\end{proof}

\begin{lemma}\label{equation}
If $m\mid n$ and $\pi(\frac{n}{m})\subseteq
\pi(m)$, then $\pi(m)=\pi(n)$.
\end{lemma}

\begin{lemma}[Zsigmondy Theorem see \cite{Z}]\label{primitiveprime}
Let $q>1$ be a natural number, and let $p$ be a primitive
prime divisor of $q^n-1$, namely $p\mid q^n-1$ but for any $1\leqslant i\leqslant n-1$, $p$
does not divide $q^i-1$. If $n>1$ and $n,q$ do not satisfy the
following two cases, then $q^n-1$ has a primitive prime divisor.

\begin{enumerate}
\item $n=2$ and $q=2^k-1$, where $k$ is a natural number.
\item $n=6, q=2$.
\end{enumerate}
\end{lemma}

\section{The proof of our Theorem}

Now we will use Theorem \ref{Gruenberg} and the lemmas introduced above to prove our main theorem. In fact, we prove it with the help of the classification theorem for finite simple groups. First of all, let us consider the case of Frobenius group $F$, which has kernel $K$ and complement $H$.
Recall that if $G$ is a finite group with $Ord(G)=Ord(F)$ then
$|G|=|F|$, $\pi_e(G)=\pi_e(F)$ and $v_n(G)=v_n(F)$ where $n\in \pi_e(F)$.  \medskip

\begin{cla}\label{cla1}
Let $F$ be a frobenius group and $G$ a finite group with $Ord(G)=Ord(F)$, then $G$ is a Frobenius group. Moreover, if $F$ is
solvable, so is $G$.
\end{cla}

\begin{proof}
We will use Gruenberg-Kegel theorem to finish the proof of Claim
\ref{cla1}.

\textbf{Case 1.} $G$ is a $2$-Frobenius group. Let $G=ABC$,
where $A$, $AB$ are normal subgroups of $G$; $AB$, $BC$ are
Frobenius group with kernels $A$, $B$ and complements $B$, $C$
respectively. By Lemmas \ref{frobenius} and \ref{2frobenius}, we have $\pi(B)=\pi_2(F)$, that is,
$\pi(B)=\pi(K)$ or $\pi(H)$. Obviously, $B$ is a cyclic group of odd order by Lemma \ref{2frobenius}.

\begin{enumerate}

\item We first assume that $\pi(B)=\pi(K)$. Since $K$ is a Hall subgroup of $F$ and
$|B|=|K|$, there is an element of $F$ with order $|B|$, and then $K$ is cyclic.
Note that $K$ is normal in $F$, so $v_{|B|}(F)=1$. On the
other hand, $AB$ is a Frobenius group with complement $B$, this implies that
$v_{|B|}(G)\geqslant v_{|B|}(AB)=|A|>1$, hence $v_{|B|}(G)>v_{|B|}(F)$, a
contradiction.

\item We next assume that $\pi(B)=\pi(H)$. Similarly as above, we can obtain that $H$ is a cyclic group, hence
$v_{|B|}(F)=|K|$. Since $A$ and $AB$ are normal subgroups of $G$, $G=ABC=C(AB)=C(BA)=CBA$, therefore any element $g\in G$ can be
written as $cba$, where $a\in A, b\in B$ and $c\in C$, and hence
$B^g=B^{cba}=((B^c)^b)^a=B^a$, this implies that
the number of conjugate subgroups of $B$ in $G$ is equal to the number of $B$ in $A$. Note that $G$ is solvable,
so all subgroups of order $|B|$ are conjugate, and then
$v_{|B|}(G)=v_{|B|}(AB)=|A|$. But $|H||K|=|F|=|G|=|A||B||C|$ and
$|B|=|H|$, thus $|A|=|K|/|C|$. Therefore,
$$v_{|B|}(F)=|K|<|K|/|C|=|A|=v_{|B|}(G),$$  we get a contradiction.

\end{enumerate}

\textbf{Case 2.} There exists a non-abelian simple group $S$ such that $S\leqslant
\overline{G}=G/N \leqslant \Aut(S)$. If $2\in\pi(H)$, then $K$ is abelian.
Let $r\in \pi(K)$, we have $|K|_r=|G|_r=|S|_r$ with the help of Theorem \ref{strongerversion}.
As $K$ is a normal abelian subgroup of $F$, the number of $r$-elements of $K$ equals
$|K|_r=|S|_r$. But the Sylow $r$-subgroup of $S$ is nontrivial,
so the number of $r$-element in $S$ is $|S|_r$, hence $S$ has only one Sylow $r$-subgroup,
that is, $S$ is not simple, it is impossible. Next if $2\in \pi(K)$, then the
Sylow $p$-subgroup of $K$ is normal in $F$ where $p\in \pi(K)$. From $Ord(F)=Ord(G)$
we can see that the Sylow $p$-subgroup of $G$ is also normal in $G$. $N=O_p(G)$, so
$p=2$ and $N$ is the normal Sylow $2$-subgroup of $G$. Thus $|G/N|$ is odd. Using
the well-known Feit-Thompson Theorem, $G/N$ is
soluble, this contradicts our assumption that $G/N$ is non-solvable.

We conclude that $G$ is a Frobenius group by Theorem \ref{strongerversion}. We now assume that $F$
is non-solvable then there exists a subgroup $H_0$ with $|H:H_0|\leqslant 2$ such that $H_0\cong Z\times \SL_2(5)$, all
Sylow subgroups of $Z$ are cyclic and $(|Z|,30)=1$ by Lemma \ref{frobenius}. This shows that
$|H|=2^3\cdot 3\cdot 5\cdot |Z|$ or $2^4\cdot 3\cdot 5\cdot |Z|$. It is not difficult to
see that $H$ has no element of order $15$ since $15\not\in \pi_e(\SL_2(5))$.
Now let $V$ be a complement of $G$, then
$\pi_e(H)=\pi_e(V)$ and $|H|=|V|$  follow from $Ord(F)=Ord(G)$. If $G$ is
solvable, then $V$ has a Hall subgroup of order $15$, this follows that
$15\in\pi_e(V)$, a contradiction. Therefore $G$ is
also non-solvable.
\end{proof}

We now consider $2$-Frobenius group. Suppose that $E$
is a $2$-Frobenius group, that is, $E=ABC$, where $A$, $AB$ are
normal subgroups of $E$; $AB$, $BC$ are Frobenius groups with
kernels $A$, $B$ and complements $B$, $C$ respectively. At the following claim we say $p$ is a prime number.\medskip

\begin{cla}\label{cla2}
Let $E$ be a $2$-Frobenius group and $G$ a finite group with
$Ord(G)=Ord(E)$, then $G$ is a $2$-Frobenius group.
\end{cla}

\begin{proof}
Let $E=ABC$ be a $2$-Frobenius group, as assumed. In the light of Claim
3.1 we assume that there is a nonabelian simple group $S$
such that $S\leqslant \overline{G}=G/N \leqslant \Aut(S)$. Since $B$ is a
cyclic Hall subgroup of $E$, all Hall subgroups of $E$ with order $|B|$ are conjugate.
Set $1\neq e=bca\in E$, where $a\in A, b\in B, c\in C$ we have $B^e=B^{bca}=B^a$. Since $AB$
is a Frobenius group with kernel $A$ and complement $B$, so the intersection of any two
Hall subgroups of $E$ of order $|B|$ is trivial and
$v_{|B|}(E)=|A|$. Obviously, $v_{|B|}(G/N)=\frac{|A|}{|N|}$ and $v_{|B|}(S)=\frac{|A|}{|N|}$.
If $S = L_2(r^f)$ where $r$ is an odd prime number and $\pi(B) = \{r\}$,
then $|B| = r^f$. Therefore, there exists an element of order $r^f$ in
$S$. Observe that the Sylow $r$-subgroup of $L_2(r^f)$ is isomorphic to
$Z_r^f$, so $f = 1$. Applying Theorem 8.2(b) and (c) of
Chapter II in \cite{H}, we have the number of Sylow
$r$-subgroups of $L_2(r)$ is $r+1$. It implies that $\frac{|A|}{|N|} =
r+1$. Referring to Table I of \cite{L} we have $G/N = S$, and so
$|A||B||C| = |S||N|$, therefore $|C| = \frac{r(r^2-1)}{2r(r+1)} =
\frac{r-1}{2}$. Using Lemma \ref{divisible}, we have $\pi(r+1) \subseteq
\pi(\frac{r-1}{2})$. From $(\frac{r+1}{2}, \frac{r-1}{2}) = 1$ we deduce that $r
= 3$, i.e., $S = L_2(3)$, this contradicts the fact that $S$ is simple.
If $S = L_3(4)$ and $\pi(B) = \{3\}$, then $S$ does not have an element of order $9$,
 which is a contradiction since $E$ has a cyclic subgroup of order $9$.

When dealing with the remaining cases of simple groups, we will use Lemma 1.1 of
\cite{MD}. If $S$ is not $L_3(4)$ with $\pi(B)=\{3\}$ or $L_2(r^f)$ with $\pi(B)=\{r\}$,
then there exists a cyclic Hall $\pi_j$-subgroup $U$ for each connected component $\pi_j$ of $S$ where $j>1$. Obviously, $|U| = |B|$.
Since $v_d(S) = v_d(E)/|N| =
|A|/|N|$ where $1 < d \in \pi_e(B)$, the intersection of any two distinct Hall
$\pi_j(S)$-subgroups is trivial in our situation, and hence they are conjugate. From the fact that Sylow subgroups are
conjugate we can obtain $v_{|B|}(S) =
|S:N_S(U)|$, that is
$$\frac{|A|}{|N|}=|S:N_S(U)|.\eqno(3.1)$$
Let $\theta = |\overline{G}/S|$. By applying $|G|=|E|$ and (3.1), we have
$$\frac{|A|}{|N|}=\frac{|S|}{|N_S(U)/U| |U|}\eqno(3.2)$$ and
$$|C|=\theta |N_S(U)/U|.\eqno(3.3)$$
In the light of Lemma \ref{divisible}, we can give an additional constraint as
follows:
$$\pi(\frac{|A|}{|N|})\subseteq\pi(|C|).\eqno(3.4)$$\medskip

Assume $S$ is an alternating group $A_n$ where $n\geqslant 5$. If the prime graph of
$S$ has $2$ components, then $n=p$, $p+1$, $p+2$, and $n$, $n-2$
are not both prime numbers. Therefore, $|U| = p$ and $|N_S(U)/U| =
\frac{p-1}{2}$. With the help of the automorphism structure of
alternating groups we have $\theta = 1$, $2$ or $4$. From (3.4) we deduce that $\pi(\frac{n!}{p(p-1)}) \subseteq
\pi(p-1)$. By Lemma \ref{equation}, we have $\pi((p-1)!) = \pi(p-1)$, so
$p=2$ or $3$, contradicting $n\geqslant 5$. If the prime
graph of $S$ has $3$ components, then $n = p$ and $p-2$ is also a
prime number. It follows that $|U| = p$ or $p-2$. We first assume that $|U| = p$, it is impossible
as in the case when the prime graph has $2$ components. We next assume that $|U| =
p-2$, we have $|N_S(U)/U| = p-3$. From equation (3.4), we can obtain
$$\pi(\frac{p!}{2(p-2)(p-3)}) \subseteq \pi(p-3).$$
 Using Lemma \ref{equation} again $\pi(\frac{p!}{2(p-2)}) = \pi(p-3)$, which is impossible.\medskip

 We now consider $S$ is a simple group of Lie type. It is clear that $U$ is a maximal torus of $S$, and thus $N_S(U)/U$ is isomorphic to a subgroup of the Weyl group $W$ of $S$. Since $|W|$ divides $|S|$, combining equations (3.2)-(3.4), we have
$$\pi(\frac{|S|}{|U||W|})\subseteq \pi(\theta_0|W|), \eqno(3.5)$$
where $\theta_0$ is a multiple of $\theta$. Note that $S$ has more than one prime graph components, then by \cite{WJS} and \cite{KA}, we can give the following Tables $1$-$3$, where $|W|$  and  the values $\theta_0$ are given in Section 3.6 of \cite{C} and listed in Tables $1$-$4$ of \cite{L}, respectively.

\begin{table}[h]\tiny 
\caption{Simple group of Lie type with disconnected prime graph and $s(S)=2$.}
\label{table I}\vspace*{-16pt}

\begin{center}
\begin{tabular}{cccccc}\\\hline
$S$ & Conditions & $|U|$ & $|W|$ & $\theta_0$ & $\frac{|S|}{|U||W|}$\\\hline

$A_{p-1}(q)$ & $(p,q)\neq (3,2),(3,4)$ & $\frac{q^p-1}{(q-1)(p,q-1)}$ & $p!$ & $2p^s$ & $\frac{q^{(^p_2)}\prod^{i=p-1}_{i=1}(q^i-1)}{p!}$\\

$A_{p}(q)$ & $q-1\mid p+1$ & $\frac{q^p-1}{q-1}$ & $(p+1)!$ & $2p^s$ & $\frac{q^{(^{p+1}_2)}(q^{p+1}-1)\prod^{i=p-1}_{i=2}(q^i-1)}{(p+1)!}$\\

$^2A_{p-1}(q)$ &  & $\frac{q^p+1}{(q+1)(p,q+1)}$ & $p!$ & $2^up^s$ & $\frac{q^{(^p_2)}\prod^{i=p-1}_{i=1}(q^i-(-1)^i)}{p!}$ \\

$^2A_{p}(q)$ & $q+1\mid p+1$ & $\frac{q^p+1}{q+1}$ & $(p+1)!$ & $2^up^s$ & $\frac{q^{(^{p+1}_2)}(q^{p+1}-1)\prod^{i=p-1}_{i=2}(q^i-(-1)^i)}{(p+1)!}$ \\

$B_n(q)$ & $n=2^m\ge 4$, $q$ odd & $\frac{q^n+1}{2}$ & $2^n\cdot n!$ & $2^s$ &  $\frac{q^{n^2}(q^{n}-1)\prod^{i=n-1}_{i=1}(q^{2i}-1)}{2^nn!}$ \\

$B_p(3)$ & & $\frac{3^p-1}{2}$ & $2^p\cdot p!$ & $1$ & $\frac{3^{p^2}(3^{p}+1)\prod^{i=n-1}_{i=}(3^{2i}-1)}{2^pp!}$ \\

$C_n(q)$ &$n=2^m\ge 4$ & $\frac{q^n+1}{(2,q-1)}$ & $2^n\cdot n!$ & $1$ & $\frac{q^{n^2}(q^{n}-1)\prod^{i=n-1}_{i=}(q^{2i}-1)}{2^nn!}$ \\

$C_p(q)$ & $q=2,3$ & $\frac{q^p-1}{(2,q-1)}$ & $2^p\cdot p!$ & $1$ & $\frac{q^{p^2}(q^{p}+1)\prod^{i=p-1}_{i=1}(q^{2i}-1)}{2^pp!}$ \\

$D_p(q)$ & $p\ge 5, q=2,3,5$ & $\frac{q^p-1}{(4,q^p-1)}$ & $2^{p-1}\cdot p!$ & $2$ & $\frac{q^{p(p-1)}\prod^{i=p-1}_{i=1}(q^{2i}-1)}{2^{p-1}p!}$ \\

$D_{p+1}(q)$ & $p\ge 3, q=2,3$ & $\frac{q^p-1}{(2,q-1)}$ & $2^{p}\cdot(p+1)!$ & $1$ & $\frac{q^{p(p+1)}(q^p+1)(q^{p+1}-1)\prod^{i=p-1}_{i=1}(q^{2i}-1)}{(2,q-1)\cdot
2^{p-1}(p+1)!}$ \\

$^2D_n(q)$ & $q=2^m\ge 4$ & $q^n+1$ & $2^{n-1}\cdot n!$ & $2^s$ & $\frac{q^{n(n-1)}\prod^{i=n-1}_{i=1}(q^{2i}-1)}{2^{n-1}n!}$ \\

$^2D_n(2)$ & $n=2^m+1\ge 5$ & $2^{n-1}+1$ & $2^{n-1}\cdot n!$ & $1$ & $\frac{2^{n(n-1)}(2^{n-1}-1)(2^n+1)\prod^{i=n-2}_{i=1}(2^{2i}-1)}{2^{n-1}n!}$ \\

$^2D_p(3)$ & $5\le p\ne 2^m+1$ & $\frac{3^p+1}{4}$ & $2^{p-1}\cdot p!$ & $2$ & $\frac{3^{p(p-1)}\prod^{i=n-1}_{i=1}(3^{2i}-1)}{2^{p-2}p!}$ \\

$^2D_n(3)$ & $9\le n=2^m+1\ne p$ & $\frac{3^{n-1}+1}{2}$ & $2^{n-1}\cdot n!$ & $1$ & $\frac{3^{n(n-1)}(3^{n-1}-1)(3^n+1)\prod^{i=n-2}_{i=1}(3^{2i}-1)}{2^{n-1}n!}$ \\

$G_2(q)$ & $2<q\equiv \epsilon (3), \epsilon=\pm 1$ & $q^2-\epsilon q+1$ & $12$ & $2^s3^u$ & $q^6(q^2-1)(q+\epsilon)(q^3-\epsilon)/12$ \\

$^3D_4(q)$ & & $q^4-q^2+1$ & $2^34!$ & $2^s3^u$ & $\frac{q^{12}(q^2-1)(q^4+q^2+1)(q^6-1)}{2^34!}$ \\

$F_4(q)$ & $q$ odd & $q^4-q^2+1$ & $2^73^2$ & $2^s3^u$ & $\frac{q^{24}(q^2-1)(q^6-1)(q^8-1)(q^8+q^6-q^2-1)}{2^73^2}$ \\

$E_6(q)$ &  & $\frac{q^6+q^3+1}{(3,q-1)}$ & $2^73^45$ & $2^s3^u$ & $\frac{q^{36}(q^2-1)(q^3-1)(q^5-1)(q^6-1)(q^8-1)(q^{12}-1)}{2^73^45}$ \\

$^2E_6(q)$ & $q>2$ & $\frac{q^6-q^3+1}{(3,q+1)}$ & $2^73^45$ & $2^s3^u$ & $\frac{q^{36}(q^2-1)(q^3+1)(q^5+1)(q^6-1)(q^8-1)(q^{12}-1)}{2^73^45}$ \\

$^2A_3(2)$ & & $5$ &$4!$ & $2$ & $2^33^3$ \\\hline
\end{tabular}
\end{center}
\end{table}

\begin{table}[h]\tiny
\caption{Simple group of Lie type with disconnected prime graph and $s(S)=3$.
}\label{table II}\vspace*{-16pt}

\begin{center}
\begin{tabular}{cccccc}\\\hline
$S$ & Conditions  & $|U|$ & $|W|$ & $\theta_0$ & $\frac{|S|}{|U||W|}$ \\\hline

$A_1(q)$ & $3<q\equiv \epsilon(\bmod~4)$ & $\frac{q+\epsilon}{2}$ & $2$  & $2^s$ & $\frac{q(q-\epsilon)}{2}$ \\

$A_1(q)$ & $q=2^f>2$ & $q\pm 1$ & $2$ & $1$ & $\frac{q(q\mp 1)}{2}$ \\

$A_1(q)$ & $q=2^r, 3^r$ & $\frac{q-1}{(2,q-1)}$ & $2$ & $r$ & $\frac{q(q-1)}{2}$ \\

$A_2(q)$ & $q>2$, $q$ even & $q\pm 1$ & $6$ & $2^s3^u$ & $\frac{q^3(q\mp 1)(q^3-1)}{6}$ \\

$^2A_5(2)$ & & $7;11$ & $6!$ & $2$ & $2^{11}3^411;2^{11}3^47$ \\

$^2D_{p}(3)$ & $p=2^m+1\ge 5$ & $\frac{3^{p-1}+1}{2}$ & $2^{p-1}p!$ & $2$ & $\frac{3^{p(p-1)}(3^{p-1}-1)(3^p+1)\prod_{i=1}^{i=p-2}(3^{2i}-1)}{2^{p-1}p!}$\\

 & & $\frac{3^{p}+1}{4}$ & & & $\frac{3^{p(p-1)}\prod_{i=1}^{i=p-1}(3^{2i}-1)}{2^{p-2}p!}$ \\

$G_{2}(q)$ & $q\equiv 0(\bmod 3)$ & $q^2\pm q+1$ & $12$ & $2^s3^u$ & $\frac{q^6(q\mp 1)(q^2-1)(q^3+1)}{12}$ \\

$^2G_{2}(q)$ & $q=3^{2m+1}>3$ & $q\pm\sqrt{3q}+1$ & $12$ & $3^s$ & $\frac{q^3(q-1)(q^3+1)}{12(q\pm\sqrt{3q}+1)}$ \\

$F_{4}(q)$ & $q$ even & $q^4+1$ & $2^73^2$ & $2^s$ & $\frac{q^{24}(q^2-1)(q^6-1)(q^4+1)(q^{12}-1)}{2^73^2}$ \\
 & & $q^4-q^2+1$ & & & $\frac{q^{24}(q^2-1)(q^6-1)(q^8-1)(q^8+q^6-q^2-1)}{2^73^2}$ \\

 & & $t_1:=q^2-\sqrt{2q^3}$ & & & $\frac{q^{12}(q-1)(q^3+1)(q^4-1)(q^6+1)}{2^73^2t_1}$ \\

$^2F_{4}(q)$ & $q=2^{2m+1}>3$ & $+q-\sqrt{2q}+1$ & $2^73^2$ & $3^s$ & \\

 & & $t_2:=q^2+\sqrt{2q^3}$ & & & $\frac{q^{12}(q-1)(q^3+1)(q^4-1)(q^6+1)}{2^73^2t_2}$\\

 & & $+q+\sqrt{2q}+1$ & & & \\

$E_7(2)$ & & $73$ & $2^{10}3^45\cdot7$ & $1$ & $2^{53}3^7 5\cdot7\cdot11\cdot13\cdot17\cdot19\cdot31\cdot43\cdot127$ \\

 & & $127$ & & & $2^{53}3^75\cdot7\cdot11\cdot13\cdot17\cdot19\cdot31\cdot43\cdot73$\\

$E_7(3)$ & & $757$ & $2^{10}3^45\cdot7$ & $1$ & $2^{13}3^{59}5\cdot7^211^213^319\cdot37\cdot41\cdot61\cdot73\cdot547\cdot1093$ \\

 & & $1093$ & & & $2^{13}3^{59}5\cdot7^211^213^319\cdot37\cdot41\cdot61\cdot73\cdot547\cdot757$ \\\hline

\end{tabular}
\end{center}
\end{table}

\begin{table}[h]\tiny
\caption{Simple group of Lie type with disconnected prime graph and $s(S)>3$.
}\label{table III}\vspace*{-16pt}

\begin{center}
\begin{tabular}{cccccc}\\\hline
$S$ & Conditions & $|U|$ & $|W|$ & $\theta_0$ & $\frac{|S|}{|U||W|}$ \\\hline

$A_2(4)$ & & $3$;$5$;$7$ & $6$ & $2$ & $2^5\cdot35;2^5\cdot21$;$2^5\cdot15$ \\

$^2B_2(q)$ & $q=2^{2m+1}>2$ & $q-1$ & $8$ & $1$ & $\frac{q^2(q^2+1)}{8}$\\

 & & $q\pm\sqrt{2q}+1$ & & & $\frac{q^2(q-1)(q^2+1)}{8(q\pm\sqrt{2q}+1)}$ \\

 & & $13$ & & & $2^{29}3^55\cdot 7^211\cdot17\cdot19$ \\

$^2E_6(2)$ & & $17$ & $2^73^45$ & $1$ & $2^{29}3^55\cdot 7^211\cdot13\cdot19$ \\

 & & $19$ & & & $2^{29}3^55\cdot 7^211\cdot13\cdot17$ \\

 & & $t_1:=\frac{q^{10}-q^5+1}{q^2-q+1}$ & & & $\frac{q^{120}(q^2-1)(q^8-1)(q^{12}-1)(q^{14}-1)(q^{20}-1)(q^{24}-1)(q^{30}-1)}{2^{14}3^55^27t_1}$ \\

$E_8(q)$ & & $t_2:=\frac{q^{10}+q^5+1}{q^2+q+1}$ & $2^{14}3^55^27$ & $2^s3^u5^v$ & $\frac{q^{120}(q^2-1)(q^8-1)(q^{12}-1)(q^{14}-1)(q^{20}-1)(q^{24}-1)(q^{30}-1)}{2^{14}3^55^27t_2}$\\

 & & $t_3:=q^8-q^4+1$ & & & $\frac{q^{120}(q^2-1)(q^8-1)(q^{12}-1)(q^{14}-1)(q^{20}-1)(q^{24}-1)(q^{30}-1)}{2^{14}3^55^27t_3}$\\

 & $q\equiv 0,1,4(\bmod 5)$ & $t_4:=\frac{q^{10}+1}{q^2+1}$ & & & $\frac{q^{120}(q^2-1)(q^8-1)(q^{12}-1)(q^{14}-1)(q^{20}-1)(q^{24}-1)(q^{30}-1)}{2^{14}3^55^27t_4}$\\  \hline
\end{tabular}
\end{center}
\end{table}

Observing Tables $1$-$3$ and the equation (3.5) we deduce that $\pi(\frac{|S|}{|U||W|})\subseteq
\pi(\theta_0|W|)$, so that if $S$ is not isomorphic to $A_1(q)$ with $\theta _0=r$, which is listed in Table $2$, then $\pi(\theta_0)\subseteq \pi(|W|)$, and hence $\pi(\frac{|S|}{|U||W|})\subseteq \pi(|W|)$.
In the light of Lemma \ref{equation}, we have $$\pi(\frac{|S|}{|U|})=\pi(|W|).\eqno(3.6)$$
Note that $\pi(\frac{|S|}{|U|})$ is the set of the characteristic and some primitive prime divisors of simple group  $S$ of Lie type.
By considering the existence of primitive prime divisors as stated in Lemma \ref{primitiveprime}, we can check equation (3.6) for Tables $1$-$3$ one by one to get that $S$ is isomorphic to $A_2(3)$, ${}^2A_2(3)$ or ${}^2A_3(2)$.
If $S={}^2A_2(3)$ then $\theta_0|N_S(U)/U|=3$, $\frac{|S|}{|N_S(U)/U||U|}=2^5\cdot 3^2$ and if $S={}^2A_3(2)$ then $\theta_0|N_S(U)/U|=8$, $\frac{|S|}{|N_S(U)/U||U|}=2^4\cdot3^4$, these two cases contradict (3.4).
Finally if $S=A_2(3)$, then $\theta_0=2$, $|N_S(U)/U|=3$ and $\frac{|S|}{|N_S(U)/U||U|}=2^4\cdot 3^2$, therefore $G/N\cong \Aut(A_2(3))$ and $|A|=2^4\cdot3^2\cdot |N|$. $|C|=6$ since $|G|=|E|$. It is easy to get that the number of elements of $G$ with order $13$ is $2^6\cdot 3^3\cdot |N|$ by some calculations with the help of Theorem \ref{strongerversion}(2)(b) and \cite{CN}. Observe that $E$ has $\phi(13)|A|=2^5\cdot 3^3\cdot |N|$ elements of order $13$, this implies that $E$ and $G$ can not be the same order type. We now suppose that $S=A_1(q)$ where $q=2^r$ or $3^r$ and $r$ is an odd prime number, then $\frac{|A|}{|N|}=\frac{q(q-1)}{2}$ and $|C|=2r$. If $q=3^r$, then by equation (3.4) we get that $3^r-1=2^e$ where $e$ is an integer, this
 follows that $r=2$, which contradicts the fact that $r$ is an odd prime number. If $q=2^r$, then by equation (3.4) again $2^r-1=r^e$ where $e$ is an integer, contracting Fermat’s little theorem, which shows that $2^r-1\equiv 1(\bmod r)$.

Finally we suppose $S$ is a sporadic simple group or ${}^2F_4(2)'$, using \cite{CN} and equations (3.2) and (3.3), we can obtain the following Table $4$.

\begin{table}[h]\scriptsize
\caption{Sporadic simple group $S$ and $^2F_4(2)'$.
}\label{table IV}

\begin{center}
\begin{tabular}{cccc}\\\hline

 $S$ & $|U|$ & $\pi(\frac{|A|}{|N|})$ & $\pi(|C|)$ \\

 $M_{11}$ & $5$;$11$ & $\{2,3,11\}$; $\{2,3\}$ & $\{2\}$;$\{2,5\}$ \\

 $M_{12}$ & $11$ & $\{2,3\}$ & $\{2,5\}$ \\

 $M_{22}$ & $5$;$7$;$11$ & $\{2,3,7,11\}$;$\{2,3,5,11\}$;$\{2,3,7\}$ & $\{2\}$;$\{2,3\}$;$\{2,5\}$ \\

 $M_{23}$ & $11$;$23$ & $\{2,3,7,23\}$;$\{2,3,5,7\}$ & $\{5\}$;$\{11\}$ \\

 $M_{24}$ & $11$;$23$ & $\{2,3,7,23\}$;$\{2,3,5,7\}$ & $\{5\}$;$\{11\}$ \\

 $J_1$ & $7$;$11$;$19$ & $\{2,5,11,19\}$;$\{2,3,7,19\}$;$\{2,5,7,11\}$ & $\{3\}$;$\{5\}$;$\{2,3\}$ \\

 $J_2$ & $7$ & $\{2,3,5\}$ & $\{2,3\}$ \\

 $J_3$ & $17$;$19$ & $\{2,3,5,19\}$;$\{2,3,5,17\}$ & $\{2\}$;$\{2,3\}$ \\

 $J_4$ & $23$;$29$ & $\{2,3,5,7,11,29,31,37,43\}$;$\{2,3,5,7,11,23,31,37,43\}$ & $\{11\};\{2,7\}$ \\

  & $31$;$37$ & $\{2,3,5,7,11,23,29,37,43\}$;$\{2,3,5,7,11,23,29,31,43\}$ & $\{5\}$;$\{2,3\}$ \\

  & $43$ & $\{2,3,5,7,11,23,29,31,37\}$ & $\{2,7\}$ \\

 $HS$ & $7$;$11$& $\{2,3,5,11\}$;$\{2,3,5,7\}$ & $\{2,3\}$;$\{2,5\}$ \\

 $Ru$ & $29$ & $\{2,3,5,13\}$ & $\{2,7\}$ \\

 $Sz$ & $11$;$13$ & $\{2,3,5,7,13\}$;$\{2,3,5,11\}$ & $\{2,5\}$;$\{2\}$ \\

 $He$ & $17$ & $\{2,3,5,7\}$ & $\{2\}$ \\

 $On$ & $11$;$19$;$31$ & $\{2,3,7,19,31\}$;$\{2,3,5,7,11,31\}$;$\{2,3,7,11,19\}$ & $\{2,5\}$;$\{2,3\}$;$\{2,3,5\}$ \\

 $McL$ & $11$ & $\{2,3,5,7\}$ & $\{2,5\}$ \\

 $Ly$ & $31$;$37$;$67$ & $\{2,3,5,7,11,37,67\}$;$\{2,3,5,7,11,31,67\}$;$\{2,3,5,7,11,31,37\}$ & $\{3\}$;$\{2,3\}$;$\{2,11\}$ \\

 $Co_1$ & $23$ & $\{2,3,5,7,13\}$ & $\{11\}$ \\

 $Co_2$ & $11$;$23$ & $\{2,3,5,7,23\}$ ; $\{2,3,5,7,11\}$ & $\{5\}$;$\{11\}$ \\

 $Co_3$ & $23$ & $\{2,3,5,7\}$ & $\{11\}$ \\

 $Fi_{22}$ & $13$ & $\{2,3,5,7,11\}$ & $\{2,3\}$ \\

 $Fi_{23}$ & $17$;$23$ & $\{2,3,5,7,11,13,23\}$;$\{2,3,5,7,13,17\}$ & $\{2\}$;$\{11\}$ \\

 $Fi'_{24}$ & $17$;$23$ & $\{2,3,5,7,11,13,23,29\}$;$\{2,3,5,7,13,17,29\}$ & $\{2\}$;$\{2,11\}$ \\

 & $29$ & $\{2,3,5,7,11,13,17,23\}$ & $\{2,7\}$ \\

 $F_1$ & $41$ & $\{2,3,5,7,11,13,17,19,23,29,31,47,59,71\}$ & $\{2,5\}$ \\

  & $59$ & $\{2,3,5,7,11,13,17,19,23,29,31,41,47,71\}$ & $\{2,7\}$ \\

  & $71$ & $\{2,3,5,7,11,13,17,19,23,29,31,41,47,59\}$ & $\{2,7\}$ \\

 $F_2$ & $31$ ; $47$ & $\{2,3,5,7,11,13,17,19,23\}$;$\{2,3,5,7,11,13,17,19,31\}$ & $\{3,5\}$;$\{2,3\}$ \\

 $F_3$ & $19$;$31$ & $\{2,3,5,7,11,13\}$;$\{2,3,5,7,11,13,19\}$ & $\{3\}$;$\{3,5\}$ \\

 $F_5$ & $19$ & $\{2,3,5,7,11\}$ & $\{2,3\}$ \\

 $^2F_4(2)'$ & $13$ & $\{2,3,5\}$ & $\{2,3\}$ \\\hline
\end{tabular}
\end{center}
\end{table}

\noindent Observing Table $4$, it can be seen that neither all sporadic simple groups nor ${}^2F_4(2)'$ satisfies equation (3.4), i.e., $$\pi(\frac{|A|}{|N|})\subseteq \pi(|C|).$$ Therefore, $G$ is a $2$-Frobenius group.
\end{proof}

\end{document}